\documentclass{amsart}
\usepackage{amsmath,amsthm,amsfonts,amssymb,amstext, latexsym, amscd}

\newtheorem{theorem}{Theorem}[section]
\newtheorem{lem}[theorem]{Lemma}
\newtheorem{prop}[theorem]{Proposition}

\newtheorem{cor}[theorem]{Corollary}

\theoremstyle{definition}
\newtheorem{definition}[theorem]{Definition}

\theoremstyle{remark}

\numberwithin{equation}{section}

\newcommand{\Hor}{{\mathcal{H}}}
\newcommand{\V}{{\mathcal{V}}}
\newcommand{\ra}{\rightarrow}

\newcommand{\C}{{\mathbb{C}}}
\newcommand{\HH}{{\mathbb{H}}}

\newcommand{\Ad}{\text{Ad}}

\newcommand{\Soul}{\Sigma}

\newcommand{\lb}{\langle}
\newcommand{\rb}{\rangle}

\newcommand{\bs}{\backslash}

\newcommand{\T}{{\rm T}}
\newcommand{\R}{\mathbb{R}}

\newcommand{\ii}{\mathbf{i}}
\newcommand{\jj}{\mathbf{j}}
\newcommand{\kk}{\mathbf{k}}

\begin{document}

\title{Metrics with nonnegative curvature on $S^2\times\R^4$}
\author{Kristopher Tapp}

\begin{abstract}
We study nonnegatively curved metrics on $S^2\times\R^4$.  First, we prove rigidity theorems for connection metrics; for example, the holonomy group of the normal bundle of the soul must lie in a maximal torus of $\text{SO}(4)$.  Next, we prove that Wilking's almost-positively curved metric on $S^2\times S^3$ extends to a nonnegatively curved metric on $S^2\times\R^4$ (so that Wilking's space becomes the distance sphere of radius $1$ about the soul).  We describe in detail the geometry of this extended metric.
\end{abstract}

\maketitle
\date{\today}

\section{Introduction and Background}
The nonnegatively curved metrics on $S^2\times\R^2$ were classified in~\cite{GT}.  Rigidity result for nonnegatively curved metrics on $S^2\times\R^3$ were obtained in~\cite{T2}, including a classification of the connection metrics.  Aside from these results, very little is known about the family of nonnegatively curved metrics on $S^n\times\R^k$.  The significance of this problem derives in part from its relationship to the generalized Hopf conjecture.  In fact, there are general relationships between the nonnegatively curved metrics on a vector bundle and on its unit sphere bundle, which we will now review.

Suppose that $M$ is an open manifold with nonnegative curvature.  According to~\cite{CG}, $M$ is diffeomorphic to the total space of the normal bundle of its soul, $\Sigma\subset M$.  We will denote this normal bundle as $\nu(\Soul)$, and its fiber at $p\in\Soul$ as $\nu_p(\Soul)$.  According to~\cite{GW}, a tube of sufficiently small radius about $\Soul$ is convex, so the tube's boundary (which can be identified with the total space of the unit sphere bundle, $\nu^1(\Soul)$, of $\nu(\Soul)$) inherits nonnegative curvature.  The following ``soul inequality'' for the curvature tensor, $R$, of $M$  is found in~\cite{T}:

\begin{prop}[\cite{T}] \label{soulin}For all $p\in\Soul$, $X,Y\in T_p\Soul$ and $V,W\in\nu_p(\Soul)$, we have:
$$(D_XR)(X,Y,W,V)^2 \leq \left( |R(W,V,X)|^2 + \frac 23 (D_XD_XR)(W,V,W,V)\right)\cdot R(X,Y,X,Y).$$
\end{prop}
In this inequality, we are considering $R$ sometimes as a function from  $(T_pM)^3\ra T_pM$ and sometimes from $(T_pM)^4\ra\R$.  The following can help decide whether a point  of $\nu^1(\Soul)$ has strictly positive curvature:
\begin{definition} A non-zero vector $V\in\nu_p(\Soul)$ is called ``good'' if the inequality of Proposition~\ref{soulin} is \emph{strictly} satisfied for all $X,Y\in T_p\Soul$ with $|X\wedge Y|\neq 0$ and all $W\in\nu_p(\Soul)$ with $|V\wedge W|\neq 0$.
\end{definition}
\begin{prop}[\cite{T}]
If $V$ is good, then for sufficiently small $\epsilon$, $\exp(\epsilon\cdot V)$ is a point of $M$ at which all planes tangent to the $\epsilon$-sphere about $\Soul$ have strictly positive curvature.
\end{prop}

The soul inequality was originally expressed in~\cite{T} in a manner which more explicitly distinguished the three different curvatures which it relates:
$$\lb (D_XR^\nabla)(X,Y)W,V\rb^2 \leq\left(|R^\nabla(W,V)X|^2 + \frac 23 (D_XD_X k^f)(W,V)\right)\cdot k_\Soul(X,Y).$$
Here, $R^\nabla$ denotes the curvature tensor of the induced connection, $\nabla$, in $\nu(\Soul)$, and $k_\Soul$  and $k^f$ denote respectively the unnormalized intrinsic sectional curvatures of $\Soul$ and of the Sharafutinov fiber, $\exp(\nu_p(\Soul))$.  This point of view leads to the idea of putting positive curvature on a sphere bundle by finding structures on the vector bundle which make the inequality strict:
\begin{prop}[\cite{T}] If structures on a Euclidean vector bundle (a metric on the base, a connection compatible with the Euclidean structure, and a smoothly varying curvature tensor on each fiber) can be found such that all unit-vectors are good, then its sphere bundle admits a metric with positive curvature.
 \end{prop}
Unfortunately, no new examples of sphere bundles with positive curvature have yet been constructed from this theorem.  The problem is a lack of existing tools for constructing structures (particularly connections) on vector bundles to satisfy this inequality.  Towards improving this situation, we believe that it is important to explicitly understand how the inequality is satisfied by known examples with nonnegative curvature.  The majority of this paper is devoted to understanding the inequality for a particular metric on $S^2\times\R^4$ -- a metric for which we'll prove that good vectors exist.

When the soul is two dimensional, $R^\nabla(X,Y)W$ does not depend on the choice of oriented orthonormal basis $\{X,Y\}$ of $T_p\Soul$, so we'll shorthand this as $R^\nabla(W)$.  With this shorthand, the inequality becomes:
$$\lb (D_XR^\nabla)(W),V\rb^2 \leq\left(\lb R^\nabla(W),V\rb^2 + \frac 23 (D_XD_X k^f)(W,V)\right)\cdot k_\Soul.$$
which is valid for all $p\in\Soul$, all $V,W\in\nu_p(\Soul)$ and all unit-length $X\in T_p\Soul$.

Our paper is organized as follows.  In Section 2, we prove rigidity results for connection metric with nonnegative curvature, including:
\begin{prop}\label{conhol} For any connection metric with nonnegative curvature on an $\R^4$ bundle over $S^2$, the holonomy group of the normal bundle of the soul  must lie in a maximal torus of $SO(4)$.  In other words, $\nu(\Soul)$ globally decomposes as two orthogonal $\nabla$-invariant $\R^2$-bundles over $S^2$.
\end{prop}
Furthermore, we prove that the restriction of the curvature tensor of $\nabla$ to one of the $\R^2$-bundles is a multiple of its restriction to the other.

In Section~3, we review Wilking's construction from~\cite{W} of an almost positively curved metric on $S^2\times S^3$.  In Section~4, we prove that his metric extends to a nonnegatively curved metric on $S^2\times\R^4$.  In the remaining sections, we prove that this extended metric has the following list of geometric features.  Identifying the fiber with  $\HH$ = the quaternions, and the base with  $S^1\bs S^3$ (where $S^3\subset\HH$ is the group of unit quaternions), we have:

\begin{prop}\textbf{(Summary of metric properties of} $\mathbf{M\cong (S^1\bs S^3)\times\mathbf{\HH}}$\textbf{)} \newline The soul of $M$ is $\Soul=\{([p],0)\mid p\in S^3\}$.  Let $q_0=([1],0)\in\Soul$, $X\in T_{q_0}\Soul\cong\text{span}\{\jj,\kk\}$ and $V\in\nu_{q_0}(\Soul)\cong\HH$ (extended to a constant section of $\nu(\Soul)$).
\begin{enumerate}
\item The distance sphere of radius $1$ about $\Soul$ is isometric to Wilking's metric with almost positive curvature on $S^2\times S^3$.
\item $S^3$ acts isometrically on $M$, with $s\in S^3$ acting as $$s\star([a],v)=([as^{-1}],svs^{-1}).$$
\item $\Soul$ has constant curvature $16$.
\item The connection, $\nabla$, in $\nu(\Soul)$ is the unique connection that is invariant under the above $S^3$-action and such that:
    $$\nabla_XV = \frac 34 XV - \frac 14 VX.$$
\item The parallel extension of $V$ along $t\mapsto([e^{tX}],0) = e^{-tX}\star q_0$ equals:
$$V(t) = e^{-\frac 34 tX}\cdot V\cdot e^{\frac 14 tX} = e^{-tX}\star\left( e^{\frac 14 tX}\cdot V\cdot e^{-\frac 34 tX}\right).$$
\item The curvature tensor, $R^\nabla$, of $\nabla$ is determined by:
$$R^\nabla(V) = \frac 72 V\ii - \frac{15}{2}\ii V.$$
\item The covariant derivative, $DR^\nabla$, of $R^\nabla$ is determined by:
$$(D_XR^\nabla)(V) = \frac{15}{8}(X\ii - \ii X)V - \frac{21}{8}V(X\ii-\ii X).$$
\item The holonomy group of $\nu(\Soul)$ is isomorphic to $SO(4)$.
\item If $V$ is not perpendicular to $1$ or $\ii$, then $V$ is a good vector.
\item There exists an orthogonal pair of 2-dimensional subspaces, $\sigma_1,\sigma_2\subset\nu_{q_0}(\Soul)$ (depending on $X$),  whose parallel extensions, $\sigma_1(t)$ and $\sigma_2(t)$, along the geodesic $\gamma(t):=([e^{tX}],0)$, satisfy $R^\nabla(\sigma_1(t))\subset\sigma_2(t)$ for all $t\in\R$.  If $X=\jj$, then  $\sigma_1=\text{span}\{1,\jj\}$ and $\sigma_2=\text{span}\{\ii,\kk\}$.
\end{enumerate}
\end{prop}

The existence of good vectors (Property 9) means that the almost positive curvature of the sphere bundle (Property 1) can be detected by second derivative information at the soul.  The good vectors are exactly those which exponentiate to points which have positive curvature in Wilking's metric.  Property (10) reflects the reason that the sphere bundle does not have strictly positive curvature.  If $V,W$ are both chosen from $\sigma_1$ (or both from $\sigma_2$), and $V(t),W(t)$ denote their parallel extensions along $\gamma(t)$, then Property (10) implies:
$$\lb R^\nabla(V(t)),W(t)\rb = 0\,\,\, \text{ and }\,\,\, \lb (D_{\gamma'(t)}R^\nabla)(V(t)),W(t)\rb = 0$$
for all $t\in\R$.  Furthermore, $(D_{\gamma'(t)}D_{\gamma'(t)} k^f)(W(t),V(t))=0$ because $R^f(V(t),W(t))$ is periodic with nonnegative second derivative, and is therefore constant. Thus, all terms of the soul inequality vanish for the triple $\{\gamma'(t),V(t),W(t)\}$.  In fact, $\sigma_1(t)$ and $\sigma_2(t)$ exponentiate to the totally geodesic flat tori in the sphere bundle which prevent the sphere bundle from having positive curvature.

We mention that Wilking's metric also extends to the vector bundle $S^3\times\R^3$, which is clear from the re-description of Wilking's metric found in~\cite{T3}.  In the final section of this paper, we prove that no structures on $S^3\times\R^3$ (or on any vector bundle over an odd-dimensional base space) could ever satisfy the soul inequality \emph{strictly} (that is, such that all vectors are good).  The analogous statement for $S^2\times\R^4$ is not known, which originally motivated our interest in understanding the extension of Wilking's metric to $S^2\times\R^4$.

The author is pleased to thank Marius Munteanu for discovering the proof of Corollary~\ref{Marius}, Sam Smith for helping with the proof of Lemma~\ref{SAM}, and Wolfgang Ziller for helpful conversations about this work.

\section{Connection Metrics}
In this section, we prove Proposition~\ref{conhol} and other rigidity results for connection metrics.  Suppose that $M$ is the total space of an $\R^4$-bundle over $S^2$ with a connection metric of nonnegative curvature.  Let $\Soul\subset M$ be a soul, let $\nu(\Soul)$ denote the normal bundle, let $\nabla$ denote the induced connection in $\nu(\Soul)$, and let $R^\nabla$ denote its curvature.

Since the metric is a connection metric, $(D_XD_X k^f)(W,V)=0$.  Regarding $R^\nabla$ at $p$ as an endomorphisms of $\nu_p(\Soul)$, the  inequality becomes:
\begin{equation}\label{soul24}\lb (D_XR^\nabla)(V),W\rb^2 \leq\lb R^\nabla(V),W\rb^2\cdot k_\Soul,\end{equation}
which is valid for all $p\in\Soul$, all $V,W\in\nu_p(\Soul)$ and all unit-length $X\in T_p\Soul$.
\begin{proof}[Proof of Proposition~\ref{conhol}]
Let $p\in\Soul$.  Choose vectors $V,W\in\nu_p(\Soul)$ for which $\lb R^\nabla(V),W\rb=0$.  Let $V(t)$ and $W(t)$ denote their parallel extensions along a piecewise geodesic, $\gamma(t)$,  in $\Soul$.  The soul inequality implies that the function $f(t): = \lb R^\nabla(V(t)),W(t)\rb$ satisfies $f'(t)^2\leq f(t)^2\cdot k_\Soul(\gamma(t))$ for all $t\in\R$.  Since $f(0)=0$, it is a simple calculus exercise to conclude that $f(t)=0$ for all $t\in\R$.  Piecewise geodesic loops at $p$ generate the normal holonomy group at $p$.  Thus, if $\lb R^\nabla(V),W\rb =0$ then $\lb R^\nabla(\Phi(V)),\Phi(W)\rb=0$ for every element, $\Phi$, in the normal holonomy group.

Since $R^\nabla$ at $p$ is a skew-symmetric endomorphism of $\nu_p(\Soul)$, we can decompose $\nu_p(\Soul)=\sigma_1\oplus\sigma_2$, where $\sigma_1$ and $\sigma_2$ are $R^\nabla$-invariant 2-dimensional subspaces, so that $\lb R^\nabla(V),W\rb=0$ for all $V\in\sigma_1$ and $W\in\sigma_2$.  This splitting is unique unless $R^\nabla_p=0$ (in which case the inequality implies that $R^\nabla=0$ at every point, so the connection is flat). In any case, we can conclude that $\sigma_1$ and $\sigma_2$ extend via parallel transport to well-defined global $\nabla$-invariant sub-bundles of $\nu(\Soul)$.
\end{proof}

The soul inequality forces additional rigidity beyond Proposition~\ref{conhol}.  Write $$\nu(\Soul) = \nu_1(\Soul)\oplus\nu_2(\Soul)$$ for the $\nabla$-invariant splitting of $\nu(\Soul)$ into a pair of $\R^2$-bundles.  For $i\in\{1,2\}$, define $F_i:\Soul\ra\R$ as $F_i(p) = \lb R^\nabla(V_i),W_i\rb$, where $\{V_i,W_i\}$ is an oriented orthonormal basis of $(\nu_i)_p(\Soul)$.  Since every $R^2$-bundle over $S^2$ is oriented, these functions are globally well-defined.

If $F_2$ vanished at a single point of $\Soul$, then the previous proof implies that it would vanish everywhere, so $\nu_2(\Soul)$ would be the trivial bundle with a flat connection.  Assuming this is not the case, we prove now that $F_1$ is a multiple of $F_2$.
\begin{prop}\label{whitlem} If $\nu_2(\Soul)$ is not flat, then there exists a constant $\lambda\in\R$ such that $F_1=\lambda F_2$.
\end{prop}
\begin{proof}
Let $p\in\Soul$, let $V,W\in\nu_p(\Soul)$ and let $X\in T_p\Soul$.  Assume $X$ is unit-length.  Decompose $V=V_1+V_2$ and $W=W_1+W_2$, where $V_i,W_i\in(\nu_i)_p(\Soul)$.  Inequality~\ref{soul24} says:
$$\lb (D_XR^\nabla)(V_1+V_2),W_1+W_2\rb^2 \leq \lb R^\nabla(V_1+V_2),W_1+W_2\rb^2\cdot k_\Soul,$$
which simplifies to:
$$\left(\lb (D_XR^\nabla)(V_1),W_1\rb + \lb (D_XR^\nabla)(V_2),W_2\rb\right) ^2
   \leq\left( \lb R^\nabla(V_1),W_1\rb + \lb R^\nabla(V_2),W_2\rb \right)^2\cdot k_\Soul.$$
Letting $c_i = V_i\wedge W_i$, this becomes:
$$\left( c_1 (XF_1) + c_2(XF_2)\right)^2\leq\left(c_1F_1 + c_2F_2\right)^2\cdot k_\Soul.$$
This inequality is valid for all $c_1,c_2$, since the vectors $V_1,W_1,V_2,W_2$ were arbitrary.  In particular, the choice $c_2 = -c_1\frac{F_1}{F_2}$ makes the right side vanish, and thus must make the left side vanish as well, which implies that
$\frac{XF_1}{F_1} = \frac{XF_2}{F_2}.$  The function $\lambda:=\frac{F_1}{F_2}$ must therefore be constant on $\Soul$, since its derivative is
$$X\lambda = \frac{F_2(XF_1) - F_1(XF_2)}{F_2^2} = \frac{0}{F_2^2} = 0.$$
\end{proof}

Since $\int_\Soul F_i$ is a topological invariant of $\nu_i(\Soul)$, the constant $\lambda$ is completely determined by the Euler classes of $\nu_1(\Soul)$ and $\nu_2(\Soul)$.  For example, if $\nu_1(\Soul)$ and $\nu_2(\Soul)$ have the same Euler classes, then $\lambda=1$.

No further restrictions can be obtained from the soul inequality.  Given a metric on $S^2$ and a connection on $S^2\times\R^4$ which satisfy the conclusions of Propositions~\ref{conhol} and Proposition~\ref{whitlem}, the soul inequality will be satisfied, provided it is separately satisfied in $\nu_1(\Soul)$ and $\nu_2(\Soul)$.

It was observed in~\cite{SW} that the total space of each nontrivial $\R^2$-bundle over $S^2$ admits a large family of nonnegatively curved connection metrics.  One obtains the simplest metric on the $k^{\text{th}}$ topological bundle type as a submersion metric of the form:
$$M_k = \left((S^3,\text{round})\times(\C,g^f)\right)/S^1,$$
where $S^1$ acts on $S^3\times\C$ as: $$e^{i\theta}\star(p,V) = (e^{i\theta}\cdot p,e^{i\theta k}\cdot V),$$ and $g^f$ is an $S^1$-invariant (that is, rotationally symmetric) metric on $\C\cong\R^2$.  The integer $k$ determines the Euler class of the resulting bundle.  For this submersion metric, the connection in the normal bundle of the soul has a parallel curvature tensor, so $DR^\nabla=0$.  One obtains a larger family of nonnegatively curved connection metrics from this one by performing (sufficiently small) arbitrary  perturbations to this connection.

The only known examples of connection metrics with nonnegative curvature on $\R^4$-bundles over $S^2$ are the submersion metrics of the form:
$$M_{(k_1,k_2)} = \left((S^3,\text{round})\times(\C,g^f_1)\times(\C,g^f_2)\right)/S^1,$$
where $S^1$ acts on $S^3\times\C\times\C$ as: $$e^{i\theta}\star(p,V_1,V_2) = (e^{i\theta}\cdot p,e^{i\theta k_1}\cdot V_1,e^{i\theta k_2}\cdot V_2),$$ and $g^f_1$ and $g^f_2$ are $S^1$-invariant (that is, rotationally symmetric) metrics on $\C\cong\R^2$.  The integers $k_1$ and $k_2$ determine the Euler classes of the resulting bundles $\nu_1(\Soul)$ and $\nu_2(\Soul)$.  For this submersion metric, $F_1$ and $F_2$ are constant functions.  Topologically, $M_{(k_1,k_2)}$ equals the Whitney sum of $M_{k_1}$ and $M_{k_2}$.  It is useful to observe:
\begin{lem}\label{SAM}
$M_{(k_1,k_2)}$ is the trivial bundle if and only if $k_1\equiv k_2(\text{mod}\,2)$.
\end{lem}
\begin{proof}
This $R^4$-bundle over $S^2$ is topologically classified by the homotopy class of its ``clutching map'' $\alpha:S^1\ra SO(4)$.  The domain of $\alpha$ is the equator of the base space, the range of $\alpha$ is the space of orthogonal maps from the fiber over the north pole to the fiber over the south pole, and the definition of $\alpha$ is $\alpha(p)$ = parallel transport along the great half-circle through $p$.  The image of $\alpha$ lies in the standard maximal torus of $SO(4)$, and is homotopic to the standard $(k_1,k_2)$-torus knot in $T^2\cong S^1\times S^1$ defined as $\alpha(t) = (e^{ik_1t},e^{ik_2t})$ with $t\in[0,2\pi]$.

We claim that $\alpha$ is nulhomotopic in $SO(4)$ if and only if $k_1\equiv k_2 (\text{mod}\,2)$.  To see this, let $\overline{\alpha}$ be a lift of $\alpha$ to the universal cover, $S^3\times S^3$, of $SO(4)$, which lies in the standard maximal torus of $S^3\times S^3$.  The derivative at the identity of the covering map between the two standard maximal tori is:  $(a,b) \mapsto (a+b,a-b)$.  Since $\alpha'(0) = (k_1,k_2)$, we know that $\overline{\alpha}'(0) = (M,N)$ with $M=(k_1+k_2)/2$  and $N=(k_1-k_2)/2$.  Thus, $\overline{\alpha}(t) = (e^{iMt},e^{iNt})$, which has period $2\pi$ (and hence closes up) if and only if $M$ and $N$ are integers, which happens if and only if $k_1\equiv k_2 (\text{mod}\,2)$.
\end{proof}

These known examples of connection metrics with nonnegative curvature on $S^2\times\R^4$ all have parallel curvature tensors.  From the above discussion, there are large families of connection metrics with non-parallel curvature tensors which satisfy the soul inequality, and it is not known whether these can be constructed to have nonnegative curvature.


\section{Wilking's metric on $S^2\times S^3$}
In this section, we summarize Wilking's construction of metrics with almost positive curvature on certain homogeneous spaces, particularly on $S^2\times S^3$.

In general, a homogeneous space $M=H\bs G$ can always be re-described as a bi-quotient of $G\times G$ as follows:
$$M=G\bs H = \Delta(G)\bs(G\times G)/(1\times H),$$
where $\Delta(G)=\{(g,g)\mid g\in G\}\subset G\times G$ denotes the diagonal.  In other words, $M$ is the quotient of $G\times G$ under the action of $\Delta(G)\times (1\times H)\cong G\times H$ defined as:
$$(g,h)\star(g_1,g_2) = (g\cdot g_1,g\cdot g_2\cdot h^{-1}).$$
The diffeomorphism from $\Delta(G)\bs(G\times G)/(1\times H)$ to $H\bs G$ sends:
\begin{equation}[g_1,g_2]\mapsto [g_2^{-1}g_1],
\label{bihom}\end{equation}
with brackets denoting equivalence classes (orbits).

The advantage of this biquotient description of $M$ is the large variety of Riemannian submersion metrics which it induces.  Any metric on $G\times G$ which is invariant under this action of $G\times H$ induces a Riemannian submersion metric on $M$.  Generally, there is a large family of such metrics on $G\times G$ which have nonnegative curvature.  Wilking discovered many examples for which the induced Riemannian submersion metric on $M$ has positive curvature almost everywhere.

His simplest such example used $G=S^3\times S^3$ and $H=S^1$ (embedded diagonally in $G$), so that topologically,
\begin{equation}M=G/H\cong T^1S^3\cong S^3\times S^2.\label{gif}\end{equation}
The biquotient description is:
\begin{eqnarray*} M = \Delta(S^3\times S^3)\bs\left((S^3\times S^3)\times(S^3\times S^3)\right)/1\times S^1.\end{eqnarray*}
The metric Wilking chose on $(S^3\times S^3)\times(S^3\times S^3)$ was a product metric, $g_L\times g_L$, where
$g_L$ is the left-invariant and right-$\Delta(S^3)$-invariant metric on $G=S^3\times S^3$ constructed as follows:
\begin{equation}(S^3\times S^3,g_L) = \left((S^3,g_0)\times(S^3,g_0)\times(S^3,g_0)\right)/S^3.\label{gl1}\end{equation}
Here, $g_0$ is bi-invariant, and $S^3$ acts by right multiplication on each of the three factors.
Notice that this quotient is diffeomorphic to $S^3\times S^3$ via:
\begin{equation}[p,v,a] \mapsto(p\cdot a^{-1},v\cdot a^{-1}).\label{yaya1}\end{equation}  Thus, $g_L$ is defined as the push-forward via this diffeomorphism of the Riemannian submersion metric on the above quotient.

In summary, $M$ is defined as the quotient of $\hat{M}:=(S^3\times S^3,g_L)\times(S^3\times S^3,g_L)$ under the action of $\hat{G}:=S^3\times S^3\times S^1$ defined as follows:
$$(g_1,g_2,\sigma)\star(a,v,b,c) = (g_1\cdot a, g_2\cdot v, g_1\cdot b\cdot\sigma^{-1}, g_2\cdot c\cdot\sigma^{-1}),$$
with the induced Riemannian submersion metric.

An explicit diffeomorphism from $M=\hat{M}/\hat{G}$ to $S^2\times S^3$ is obtained using Equation~\ref{bihom} and an explicit formula for the identification in Equation~\ref{gif} as follows:
\begin{alignat}{3}
\hat{M}/\hat{G}\,\,\,\,\,\,\,\,\,          &\cong     & \,\,\,\,\,\,\,S^1\bs(S^3\times S^3) \,\,\,\,\,\,\,\,      &\cong    & (S^1\bs S^3)\times S^3 \,\,\,\,\,\,\,\,\,\,\,\,\,\, \notag \\
               &             &                   [x,y]\,\,\,\,\,\,\,\,\,\,\,\,\,\,\,   &\mapsto & ([x],x^{-1}y)\,\,\,\,\,\,\,\,\,\,\,\,\,\,\, \notag\\
[a,v,b,c] \,\,\,\,\,&\mapsto & [b^{-1}\cdot a, c^{-1} v]\,\,\,\,\,\,\, &\mapsto & \,\,\,\,\,\,\,([b^{-1}\cdot a], a^{-1}\cdot b\cdot c^{-1}\cdot v).\label{diffeos1}
\end{alignat}

Since $(S^3\times S^3,g_L)$ is itself a quotient of $S^3\times S^3\times S^3$, we can re-express the space $M$ as a quotient of the space
$$\overline{M}=(S^3\times S^3\times S^3)\times(S^3\times S^3\times S^3)$$
(with the product metric in which each $S^3$ has a unit-round metric).  More precisely, $M$ is the quotient of $\overline{M}$ by the free action of the Lie group $\overline{G}=S^3\times S^3\times S^3\times S^3\times S^1$ defined as follows:
$$(g_1,g_2,s,t,\sigma)\star(a,v,x,b,c,y) = (g_1as^{-1},g_2vs^{-1},xs^{-1},g_1bt^{-1},g_2ct^{-1},\sigma y t^{-1}),$$
with the induced Riemannian submersion metric.  This is called the ``normal biquotient'' description of $M$.
An explicit diffeomorphism from $\overline{M}/\overline{G}$ to $S^2\times S^3$ is obtained by combining Equation~\ref{yaya1} and~\ref{diffeos1}, as follows:
\begin{alignat}{3}
\overline{M}/\overline{G} \,\,\,\,\,         &\,\,\,\cong\,\,\,     & \hat{M}/\hat{G}\hspace{.3in}\,\,\,\,\,      &\,\,\,\cong \,   & (S^1\bs S^3)\times S^3 \hspace{.3in}\label{secondlong}\\
 [a,v,x,b,c,y]              &  \,\,\,\mapsto\,\,\,       &     [ax^{-1}, vx^{-1},by^{-1},cy^{-1}]   &\,\,\,\mapsto \,\,\, & ([yb^{-1}ax^{-1}],xa^{-1}bc^{-1}vx^{-1}).\notag
\end{alignat}

\section{Extending Wilking's metric to $S^2\times\HH$}
In this section, we extend Wilking's metric on $S^2\times S^3$ to a metric on the trivial vector bundle $S^2\times\HH$.

We first establish notation related the quaternions, $\HH=\text{span}\{1,\ii,\jj,\kk\}$.  We will always consider $S^3$ as the group of unit-length elements of $\HH$, with circle-subgroup $S^1=\{e^{\ii t}\mid t\in\R\}\subset S^3\subset\HH$.  If $v\in\HH$, then $L_v,R_v:\HH\ra\HH$ will denote the left and right multiplication maps.  If $v\in S^3$, then we denote $\Ad_v = L_v\circ R_{v^{-1}}$, which restricts to $\text{Im}(\HH)\cong sp(1)$ as the usual adjoint action.  The real and imaginary parts of $v\in\HH$ will be denoted as $\text{Re}(v)$ and $\text{Im}(v)$.  The conjugate will be denoted as $\overline{v}:=\text{Re}(v)-\text{Im}(v)$.
Finally, recall that the (real) inner product of $q_1,q_2\in\HH$ is:
\begin{equation}\label{meatball}\lb q_1,q_2\rb = \text{Re}(q_1\overline{q}_2) = \text{Re}(\overline{q}_1q_2) = \text{Re}(q_2\overline{q}_1) =\text{Re}(\overline{q}_2q_1).\end{equation}

We will modify the definition of $M$ from the previous chapter by replacing one occurrence of ``$S^3$'' with ``$\HH$''.  The non-normal description of this modification is:
$$M=\Delta(S^3\times S^3)\bs\left((S^3\times\HH,g_L')\times(S^3\times S^3,g_L)\right)/1\times S^1,$$
where $g_L'$ is defined like $g_L$ by replacing one occurrence of ``$(S^3,g_0)$'' with ``$(\HH,\text{flat})$'' in Equation~\ref{gl1}

However, all future calculations will instead be done using the equivalent modification of the normal description of $M$.  That is, we define
$$\overline{M}=(S^3\times\HH\times S^3)\times(S^3\times S^3\times S^3),$$
with the product metric in which each $S^3$ has a unit-round metric and $\HH\cong\R^4$ has a flat metric.  Then we define $M$ to be the quotient of $\overline{M}$ by the action of $\overline{G}=S^3\times S^3\times S^3\times S^3\times S^1$ defined as follows:
\begin{equation}(g_1,g_2,s,t,\sigma)\star(a,v,x,b,c,y) = (g_1as^{-1},g_2vs^{-1},xs^{-1},g_1bt^{-1},g_2ct^{-1},\sigma y t^{-1}),\label{action}\end{equation}
with the induced metric which makes the projection $\pi:\overline{M}\ra M$ become a Riemannian submersion.

Define $\Sigma:=\{[a,v,x,b,c,y]\in M\mid v=0\}$, which is the soul of $M$.  The Sharafutdinov projection, $\text{sh}:M\ra\Sigma$, is the map which send $[a,v,x,b,c,y]\mapsto[a,0,x,b,c,y]$.  Notice that Wilking's space (described in the previous chapter) is isometric to the sphere of radius $1$ about $\Soul$ in $M$.  Let $\nu(\Soul)$ denote the normal bundle of $\Soul$ in $M$.  If $p\in\Soul$, then let $\nu_p(\Soul)$ denote its fiber at $p$.

Exactly as in Equation~\ref{secondlong}, an explicit diffeomorphism from $M=\overline{M}/\overline{G}$ to $(S^1\bs S^3)\times\HH\cong S^2\times\HH$ is given by:
\begin{equation} [a,v,x,b,c,y]  \mapsto ([yb^{-1}ax^{-1}],xa^{-1}bc^{-1}vx^{-1}).\label{diffeoo}\end{equation}

%
Let $\V$ and $\Hor$ denote the vertical and horizontal distributions of the Riemannian submersion $\pi:\overline{M}\ra M$.  It is straightforward to verify that:
\begin{lem}\label{Hq0} At the point $\overline{q}_0=(1,0,1,1,1,1)\in\overline{M}$, the horizontal space is:
$$\Hor_{\overline{q}_0} = \{(A,B,-A,-A,0,A)\mid A\in\text{span}\{\jj,\kk\}\subset\HH, B\in\HH\}.$$
\end{lem}

Later, we will describe $\Hor$ at an arbitrary point of $\overline{M}$, but first we will establish several consequences of this special case.
\section{Isometries of $M$}
The group $S^3$ acts by isometries on $M$, with $s\in S^3$ acting as:
\begin{equation}s\star[a,v,x,b,c,y] = [a,v,sx,b,c,y].\label{symmetry}\end{equation}
If $M$ is identified with $(S^1\bs S^3)\times\HH$ via diffeomorphism~\ref{diffeoo}, then this isometric $S^3$-action on $M$ looks like:
\begin{equation}\label{S3identify}s\star([a],v) = ([as^{-1}],svs^{-1}).\end{equation}
This action restricts to the soul, $\Soul\cong\{([q],0)\mid q\in S^3\}$, as the standard transitive action of $S^3$ on $S^2$, so the soul is homogeneous, and therefore round.  In fact,
\begin{lem}\label{curvsoul}  The soul, $\Soul$, of $M$ has constant curvature $16$.
\end{lem}
\begin{proof} If $X\in\text{span}\{\jj,\kk\}\subset\HH$ has unit-length (with respect to inner product~\ref{meatball}), then the geodesic in $\overline{M}$ given by
$$\overline{\gamma}(t) = (e^{(t/4)X},0,e^{-(t/4)X},e^{-(t/4)X},1,e^{(t/4)X})$$
is initially horizontal by Lemma~\ref{Hq0}, and is therefore everywhere horizontal.  This geodesic has constant speed equal to $|\overline{\gamma}'(0)| = \frac 12$, so the path $\gamma:=\pi\circ\overline{\gamma}$ in $\Soul\subset M$ also has constant speed $\frac 12$.  The image of $\gamma$ in $(S^1\bs S^3)\times\HH$ under diffeomorphism~\ref{diffeoo} is the path $t\mapsto ([e^{tX}],0)$.  The period of $\gamma$ equals the period of this image, which is $\pi$.  Thus, $\Soul$ has a homogeneous metric with a geodesic of speed $\frac 12$ and period $\pi$.
\end{proof}
\section{The normal connection}
In this section, we describe the connection, $\nabla$, in normal bundle, $\nu(\Soul)$, of $\Soul$ in $M$.  We continue to identify $M\cong (S^1\bs S^3)\times\HH$ and $\Soul\cong\{([q],0)\mid q\in S^3\}$ and $q_0=([1],0)\in\Soul$.  Since this identification provides an explicit trivialization of the bundle, we can identify $\nabla$ with its ``connection difference form,'' compared to the trivial flat connection in $(S^1\bs S^3)\times\HH$.  In other words, for each $q\in S^3$, $X\in\T_{([q],0)}\Soul$ and $V\in\nu_{([q],0)}(\Soul)\cong\HH$, the expression $\nabla_X V$ will denotes the covariant derivative at $([q],0)$ of the \emph{constant} extension of $V$ along a path in the direction of $X$.
\begin{prop}\label{connectionn}
$\nabla$ is the unique connection on the trivial bundle $(S^1\bs S^3)\times\HH$ with the following two properties:
\begin{enumerate}
\item For any $X\in T_{q_0}\Soul=\text{span}\{\jj,\kk\}$ and any $V\in\HH\cong\nu_{q_0}(\Soul)$,
 $$\nabla_XV = \frac 34 XV - \frac 14 VX.$$
\item The $S^3$-action of Equation~\ref{S3identify} leaves $\nabla$ invariant, which means that
$$\nabla_{(s\star X)} (s\star V) = s\star(\nabla_X V)$$
for all $q\in S^3$, $X\in T_{([q],0)}\Soul$, $V\in\HH\cong\nu_{([q],0)}(\Soul)$, and $s\in S^3$ (with $s$ acting on vectors via the derivative of the isometry it represents).
\end{enumerate}
\end{prop}
\begin{proof} Property (2) is obvious because the $S^3$-action is by isometries.  It is straightforward to check that the equation in property (1) is isotropy-invariant, and therefore that it extends to a \emph{well-defined} connection, which is clearly unique.  So it remains to establish property (1).

For this, let $X\in T_{([1],0)}\Soul=\text{span}\{\jj,\kk\}$.  As in the proof of Lemma~\ref{curvsoul}, the horizontal geodesic
$$\overline{\gamma}(t):=\left(e^{(t/4)X},0,e^{(-t/4)X},e^{(-t/4)X},1,e^{(t/4)X}\right)$$
in $\overline{M}$ is such that the geodesic $\gamma=\pi\circ\overline{\gamma}$ in $M$ is identified with the geodesic $\gamma(t)\cong ([e^{tX}],0)$ in $(S^1\bs S^3)\times\HH$.

Next, let $v\in\HH$, and consider the following path in $\overline{M}$:
$$\overline{v}(t):=\left(e^{(t/4)X},e^{(3t/4)X} v e^{(-t/4)X},e^{(-t/4)X},e^{(-t/4)X},1,e^{(t/4)X}\right)$$
Notice that $v(t):=\pi(\overline{v}(t))$ is a path in $M$ which is identified in $(S^1\bs S^3)\times\HH$ with the path $v(t)\cong ([e^{tX}],v)$.

Let $\overline{V}(t)$ be the vector field along $\overline{\gamma}(t)$ in $\overline{M}$ which exponentiates to $\overline{v}(t)$; namely,
$$\overline{V}(t):=\left(0,e^{(3t/4)X} v e^{(-t/4)X},0,0,0,0\right),$$
so that $V(t):=\pi_*(\overline{V}(t))$ is the vector field along $\gamma(t)$ in $M$ which exponentiates to $v(t)$.  Since $\overline{V}(t)$ is a horizontal vector field along a horizontal geodesic, we have:
$$V'(0) = \pi_*\left(\overline{V}'(0)\right)=\pi_*\left(0,\frac 34 Xv - \frac 14 vX,0,0,0,0\right).$$
Under the identification $M\cong (S^1\bs S^3)\times\HH$, this vector $V'(0)\in T_{q_0} M$ is identified with:
$$V'(0)\cong\left(0,\frac 34 Xv - \frac 14 vX\right)\in T_{q_0}\left((S^1\bs S^3)\times\HH\right).$$
\end{proof}
\begin{cor}\label{parallel} If $X\in\text{span}\{\jj,\kk\}$ and $V\in\nu_{q_0}(\Soul)\cong\HH$, then parallel extension of $V$ along the path $\gamma(t)=([e^{tX}],0)$ is:
$$V(t) = e^{-\frac 34 tX}\cdot V\cdot e^{\frac 14 tX} = e^{-tX}\star\left( e^{\frac 14 tX}\cdot V\cdot e^{-\frac 34 tX}\right).$$
\end{cor}
\begin{cor} The holonomy group of $\nu(\Soul)$ is isomorphic to $SO(4)$.
\end{cor}
\begin{proof}
Let $X\in\{\jj,\kk\}$, let $Y$ denote the orthogonal compliment of $X$ in $\text{span}\{\jj,\kk\}$ and let $\gamma(t) = ([e^{tX}],0)$, which is a closed geodesic in $\Soul$ with period $\pi$.  The parallel transport map, $P_\gamma$, along one iteration of $\gamma$, is the endomorphism of $\nu_{q_0}(\Soul)\cong\HH$ defined by
$P_\gamma(V) = e^{-\frac 34\pi X}\cdot V\cdot e^{\frac 14\pi X}$.  It is straightforward to check that
$$P_\gamma(1)=-X,\,\,\,\,P_\gamma(X)=1,\,\,\,P_\gamma(\ii)=-\ii,\,\,\,\,P_\gamma(Y)=-Y.$$
In particular, $\text{span}\{1,X\}$ and $\text{span}\{\ii,Y\}$ are the irreducible subspaces of $\HH$.

The holonomy group, $\Phi$, is a subgroup of $SO(4)$.  It is not contained in a maximal torus of $SO(4)$ because the irreducible subspaces for $P_\gamma$ vary with $X$.  Further, $\Phi$ is not isomorphic to $SO(3)$ or $S^3$ acting non-transitively on $\HH$ because $P_\gamma$ does not have any fixed vectors. Finally, $\Phi$ is not isomorphic to $SO(3)$ or $S^3$ acting transitively on $\HH$ because the isotropy groups are too large; for example, there are infinitely many different holonomy elements sending $\ii\mapsto-\ii$ (namely, $P_\gamma$ for all $X$).  The only remaining possibility is that $\Phi$ is all of $SO(4)$.
\end{proof}
Next we study the curvature, $R^{\nabla}$, of $\nabla$ at $q_0$.  For $V\in\HH\cong\nu_{q_0}(\Soul)$, the expression
$$R^{\nabla}(V):=R^\nabla(X,Y)V$$
does not depend on the choice of oriented orthonormal basis, $\{X,Y\}$, of $T_{q_0}\Soul\cong\text{span}\{\jj,\kk\}$.  By the proof of Lemma~\ref{curvsoul}, the basis $\{2\jj,2\kk\}$ is orthonormal, so $$R^{\nabla}(V)=4R^\nabla(\jj,\kk)V.$$

\begin{lem}\label{mncurv} For all $V\in\HH$,
$R^\nabla(V) = \frac{7}{2} V\ii - \frac{15}{2}\ii V.$
\end{lem}
\begin{proof}
Over a neighborhood of $q_0$ in $\Soul$, extend $V$ to the \emph{constant} section, and let $2X$ and $2Y$ denote the extensions of $2\jj$ and $2\kk$ to coordinate vector fields of a \emph{normal} coordinate patch at $q_0$.  We must compute the following at $q_0$:
$$R^\nabla(V) = 4R^\nabla(\jj,\kk)V = 4(\nabla_X\nabla_Y V - \nabla_Y\nabla_X V).$$
Consider the one parameter group $a(t)=e^{tX}$ in $S^3$, and the geodesic $\gamma(t)=a(-t)\star q_0 = ([a(t)],0)$ in $\Soul$.  let $Y(t):=Y(\gamma(t))$ and $V(t):=V(\gamma(t))$ denote the restrictions of $Y$  and $V$ to $\gamma$. The value of $\nabla_YV$ at the point $\gamma(t)$ equals:
\begin{eqnarray*}
(\nabla_YV)(\gamma(t)) & = & a(-t)\star\left(\nabla_{(a(t)\star Y(t))} (a(t)\star V(t))\right) \\
                       & = & a(-t)\star\left(\nabla_{g(t)Y}\left(a(t)Va(-t)\right)\right),\text{ where }g(0)=1\text{ and }g'(0)=0\\
                       & = & g(t)\cdot a(-t)\star\left(\frac 34 Y a(t)Va(-t)-\frac 14 a(t)Va(-t)Y  \right)\\
                       & = & g(t)\left(\frac 34 a(-t) Y a(t)V -\frac 14 Va(-t)Y a(t)  \right).
\end{eqnarray*}
The covariant derivative of the section $t\mapsto(\nabla_YV)(\gamma(t))$ along $\gamma(t)$ at $\gamma(0)=q_0$ is now found by adding its covariant derivative with respect to the flat connection to the connection difference form:
\begin{eqnarray*}
(\nabla_X\nabla_Y V)_{q_0}
                    & = & \frac{d}{dt}\Big|_{t=0}\left( g(t)\left(\frac 34 a(-t) Y a(t)V -\frac 14 Va(-t)Y a(t)  \right)\right)\\
                    &   & + \frac 34 X\left(\frac 34 YV-\frac 14 VY\right) - \frac 14\left(\frac 34 YV-\frac 14 VY\right) X \\
                    & = & -\frac 34(XY-YX)V + \frac 14V(XY-YX)\\
                    &   & + \frac{9}{16} \ii V-\frac{3}{16} (XVY+YVX) - \frac{1}{16}V\ii \\
                    & = & -\frac{15}{16} \ii V-\frac{3}{16} (XVY+YVX) +\frac{7}{16}V\ii .
\end{eqnarray*}
The result follows by similarly computing $(\nabla_Y\nabla_X V)_{q_0}$ and subtracting.
\end{proof}
\begin{cor}\label{C:splitting} Let $\{X,Y\}$ be an orthonormal basis of $\text{span}\{\jj,\kk\}$, and consider the splitting: $\HH=\sigma_1\oplus\sigma_2$, where $\sigma_1=\text{span}\{1,X\}$ and $\sigma_2=\text{span}\{\ii,Y\}$.  Let $\sigma_1(t)$ and $\sigma_2(t)$ denote the parallel extensions of these planes along $\gamma(t)=([e^{tX}],0)$.
\begin{enumerate}
\item For any $t\in\R$, $R^\nabla(\sigma_1(t))\subset\sigma_2(t)$ and $(D_{\gamma'(t)}R^\nabla)(\sigma_1(t))\subset\sigma_2(t)$.
\item If $V(t),W(t)\in\sigma_1(t)$ (or $V(t),W(t)\in\sigma_2(t)$), then all terms of the soul inequality vanish for the triple $\{V(t),W(t),\gamma'(t)\}$.
\end{enumerate}
\end{cor}
\begin{proof}
Lemma~\ref{mncurv} implies that $\text{span}\{1,\ii\}$ and $\text{span}\{\jj,\kk\}$ are the invariant subspaces for $R^\nabla$ at $q_0$.
Therefore, $R^\nabla(\sigma_1)\subset\sigma_2$.  Furthermore, Corollary~\ref{parallel} implies that $e^{-tx}\star\sigma_i = \sigma_i(t)$ for $i=1,2$.  In other words, there is an isometry taking $\sigma_i$ to $\sigma_i(t)$, so
$R^\nabla(\sigma_1(t))\subset\sigma_2(t)$ for all $t\in\R$.  It follows that $(D_{\gamma'(t)}R^\nabla)(\sigma_1(t))\subset\sigma_2(t)$ for all $t\in\R$ as well.

To prove part (2), let $V,W\in\sigma_1$ (or $V,W\in\sigma_2$) and let $V(t),W(t)$ denote their parallel extensions along $\gamma(t)$.  Since $\lb R^\nabla(V(t)),W(t)\rb = 0$ for all $t\in\R$, the soul inequality implies that the periodic function $t\mapsto k^f(V(t),W(t))$ has nonnegative second derivative, and must therefore be constant.  Thus, all terms of the soul inequality vanish.
\end{proof}
If the metrics on the Sharafutdinov fibers were modified in any manner which maintained nonnegative curvature, the above proof would remain valid, so we would still have, for every closed geodesic in the soul, a pair of parallel planes along along which all terms of the soul inequality would be forced to vanish.

Finally, we compute the covariant derivative, $DR^\nabla$, of the tensor $R^\nabla$.
\begin{lem}\label{L:DR} For all $V\in\HH$ and all $X\in T_{q_0}\Soul=\text{span}\{\jj,\kk\}$,
$$(D_X R^\nabla)(V) = \frac{15}{8}(X\ii-\ii X)V - \frac{21}{8}V(X\ii-\ii X).$$
\end{lem}
\begin{proof}
Consider the one parameter group $a(t)=e^{tX}$ in $S^3$, and the corresponding geodesic $\gamma(t)=a(-t)\star q_0 = ([a(t)],0)$ in $\Soul$.
Extend $V$ to a section along $\gamma$ as follows:
$$V(t) = a(-t)\star V = a(-t)Va(t).$$
The covariant derivative of $V(t)$ equals its covariant derivative with respect to the flat connection plus the connection difference form; that is,
$$\frac{D}{dt}\Big|_{t=0} V(t) = (-XV+VX) + \left(\frac 34XV-\frac 14VX\right) = -\frac 14 XV+\frac 34 VX.$$
So we have:
\begin{eqnarray*}
(D_XR^\nabla)(V)
  & = & \frac{D}{dt}\Big|_{t=0}\left(R^\nabla(V(t))\right) - R^\nabla\left(\frac{D}{dt}\Big|_{t=0} V(t)\right) \\
  & = & \frac{D}{dt}\Big|_{t=0}\left(a(-t)\star R^\nabla(V)\right) - R^\nabla\left(-\frac 14 XV+\frac 34VX\right)\\
  & = & -\frac 14XR^\nabla(V) +\frac 34 R^\nabla(V)X - R^\nabla\left(-\frac 34XV+\frac 34 VX\right)\\
\end{eqnarray*}
Using Lemma~\ref{mncurv}, this simplifies to the desired formula.
\end{proof}
\section{The horizontal distribution and the O'Neill $A$-tensor}
In this section, we derive a formula for the horizontal space of $\pi:\overline{M}\ra M$ at an arbitrary point of $\overline{M}$, and then use this formula to compute the O'Neill $A$-tensor at the point $\overline{q}_0=(1,0,1,1,1,1)$.

\begin{prop}
Let $a,b,c,v,x,y\in S^3$, $r> 0$, and $\overline{q}:=(a,rv,x,b,c,y)\in\overline{M}$.
\begin{eqnarray*} \V_{\overline{q}}
  & = & \{(R_aX,R_{rv}Y,0,R_bX,R_cY,0)\mid X,Y\in sp(1)\} \\
  &   & \oplus\{(L_a S,L_{rv}S,L_xS,L_b T,L_cT,L_yT)\mid S,T\in sp(1)\} \\
  &  & \oplus\text{span}\{(0,0,0,0,0,R_y\ii)\},\\
\text{and}\,\,\,\,\Hor_{\overline{q}} & = & \{\mathfrak{X}_{\overline{q}}(A,B)\mid A\in\text{span}\{\jj,\kk\}, B\in\HH\},
\end{eqnarray*}
where $\mathfrak{X}_{\overline{q}}(A,B)$ is defined as follows:
\begin{eqnarray*}
\mathfrak{X}_{\overline{q}}(A,B) & := &
   \Bigg(R_a(\Ad_{by^{-1}}(A-r\cdot\text{Im}(B))), R_{rv}\left(\frac 1r\Ad_{cy^{-1}}B\right),\\
  &  & L_x\left(\Ad_{a^{-1}by^{-1}}(-A+r\cdot\text{Im}(B)) - r\cdot\Ad_{v^{-1}cy^{-1}}\text{Im}(B)\right),\\
  &  & L_b\left(\Ad_{y^{-1}}(-A+r\cdot\text{Im}(B))\right),-r\cdot L_c\left(\Ad_{y^{-1}}\text{Im}(B)\right),R_yA\Bigg).
\end{eqnarray*}
\end{prop}
\begin{proof}
The formula for the vertical space is easily verified, because it is spanned by the action fields.  It is straightforward to verify that each vector in the alleged horizontal space is orthogonal to the vertical space.
Since $$\text{dim}(\Hor_{\overline{q}})=\text{dim}(\overline{M}) - \text{dim}(\V_{\overline{q}}) = 19-13 = 6,$$ the alleged horizontal space must equal the entire horizontal space.
\end{proof}
For fixed vectors $A\in\text{span}\{\jj,\kk\}$ and $B\in\HH$, we can regard $\overline{q}\mapsto\mathfrak{X}_{\overline{q}}(A,B)$ as a smooth horizontal vector field on the $r\neq 0$ portion of $\overline{M}$, but this field does not extend continuously to the $r=0$ portion of $\overline{M}$ (although Lemma~\ref{Hq0} can be re-proven by considering the limit as $r\ra 0$).

We next describe the O'Neill A-tensor of $\pi$ at the point $\overline{q}_0=(1,0,1,1,1,1)\in\overline{M}$.
\begin{prop} Let $\overline{q}_0=(1,0,1,1,1,1)\in\overline{M}$.  Suppose $X,A\in\text{span}\{\jj,\kk\}$ and $V,W\in\HH$, so that $\overline{\mathcal{X}}=(X,V,-X,-X,0,X)$ and $\overline{\mathcal{Y}}=(A,W,-A,-A,0,A)$ both lie in the horizontal space $\Hor_{\overline{q}_0}$.
\begin{enumerate}
\item $X=0$, then:
$$
A(\overline{\mathcal{X}},\overline{\mathcal{Y}})  =
  \left(-\hat{R}+S,0,\hat{R}-\hat{L}-S,\hat{R}-S,-\hat{R},S\right),
$$
where $\hat{R}:=\text{Im}(W\overline{V})$, $\hat{L}:=\text{Im}(\overline{V}W)$ and $S:=\frac 14\left(\lb 3\hat{R}-\hat{L},\jj\rb\jj +\lb 3\hat{R}-\hat{L},\kk\rb\kk\right)$.
\item If $V=0$ then,
$$
A(\overline{\mathcal{X}},\overline{\mathcal{Y}})  =
  \left( -\frac 52[X,A],0,\frac 72[X,A],\frac 32[X,A],0,-\frac 12[X,A]\right).
$$
\end{enumerate}
\end{prop}
\begin{proof}
The most efficient way to explicitly project a vector $$\overline{\mathcal{W}}=(W_1,W_2,W_3,W_4,W_5,W_6)\in T_{\overline{q}_0}\overline{M}$$ onto the vertical space $\V_{\overline{q}_0}$ is to subtract its horizontal component as follows:
$$\overline{\mathcal{W}}^\V = \overline{\mathcal{W}} - (0,W_2,0,0,0,0) - \lb \overline{\mathcal{W}},\overline{H}_1\rb \overline{H}_1 - \lb \overline{\mathcal{W}},\overline{H}_2\rb \overline{H}_2,$$
where
$$\overline{H}_1 = \frac 12 (\jj,0,-\jj,-\jj,0,\jj) \,\,\,\text{ and }\,\,\, \overline{H}_2 = \frac 12 (\kk,0,-\kk,-\kk,0,\kk).$$

We first assume that $X=0$ and prove part (1).  We lose no generality in assuming that $|V|=1$, in which case $V^{-1}=\overline{V}$, so $\hat{R}=\text{Im}(R_{V^{-1}}W)$ and $\hat{L}=\text{Im}(L_{V^{-1}}W)$.  The path $r\mapsto\overline{\gamma}(r): = (1,r\cdot V,1,1,1,1)$ in $\overline{M}$ goes through $\overline{\gamma}(0)=\overline{q}_0$ with initial direction $\overline{\gamma}'(0)=\overline{\mathcal{X}}$.  The following is a differentiable extension of the horizontal vector $\overline{\mathcal{Y}}$ to a horizontal field, $\overline{\mathcal{Y}}(r)$, along this path:
$$\overline{\mathcal{Y}}(r) = \mathfrak{X}_{\overline{\gamma}(r)}(A,R_{V^{-1}}W) = \left(A-r\hat{R},W, -A+r(\hat{R}-\hat{L}),-A+r\hat{R},-r\hat{R},A\right).$$
So we have:
\begin{eqnarray*}
A(\overline{\mathcal{X}},\overline{\mathcal{Y}}) & = & \left(\frac{d}{dr}\Big|_{r=0}\mathcal{Y}(r)\right)^\V
              =\left(-\hat{R},0,\hat{R}-\hat{L},\hat{R},-\hat{R},0\right)^\V.
\end{eqnarray*}
The formula in part (1) of the proposition is obtained by explicitly computing the $\V$-projection in the manner described above.

Next, we assume that $V=0$, and prove part (2).  In this case, the path $t\mapsto\overline{\gamma}(t):=\left(e^{tX},0,e^{-tX},e^{-tX},1,e^{tX}\right)$ in $\overline{M}$ goes through $\overline{\gamma}(0)=\overline{q}_0$ with initial direction $\overline{\gamma}'(0)=\overline{\mathcal{X}}$.  The following is a differentiable extension of the horizontal vector $\overline{\mathcal{Y}}$ to a horizontal field, $\overline{\mathcal{Y}}(t)$, along this path:
\begin{eqnarray*}
\overline{\mathcal{Y}}(t) & = & \lim_{r\ra0}\mathfrak{X}_{\left(e^{tX},r\cdot 1,e^{-tX},e^{-tX},1,e^{tX}\right)}(A,B) \\
               & = & \left(R_{e^{tX}}(\Ad_{e^{-2tX}}A),\Ad_{e^{-tX}}B,L_{e^{-tX}}(\Ad_{e^{-3tX}}(-A)),L_{e^{-tX}}(\Ad_{e^{-tX}}(-A)), 0,R_{e^{tX}}A\right).
\end{eqnarray*}
So we have:
\begin{eqnarray*}
A(\overline{\mathcal{X}},\overline{\mathcal{Y}}) & = & \left(\frac{d}{dt}\Big|_{t=0}\mathcal{Y}(t)\right)^\V
              =\left(-\frac 52[X,A],-[X,B],\frac 72[X,A],\frac 32[X,A],0,-\frac 12[X,A]\right)^\V.
\end{eqnarray*}
The formula in part (2) of the proposition is obtained by explicitly computing the $\V$-projection in the manner described above, using that $[X,A]\in\text{span}\{\ii\}$.
\end{proof}

\section{Vertical curvatures}
In this section, we use O'Neill's formulas to study the sectional curvature of an arbitrary plane at $q_0$ spanned by two normal vectors to the soul.

As in the previous section, define $\overline{q}_0=(1,0,1,1,1,1)\in\overline{M}$ and $q_0=\pi(\overline{q}_0)\in\Sigma$.  Let $V,W\in\HH$ be orthogonal and unit-length, set $\overline{\mathfrak{V}}=(0,V,0,0,0,0)$ and set $\overline{\mathfrak{W}} = (0,W,0,0,0,0)$, so that $\{\mathfrak{V}:=\pi_*\overline{\mathfrak{V}},\mathfrak{W}:=\pi_*\overline{\mathfrak{W}}\}$ is a general orthonormal pair in $\nu_{q_0}(\Soul)$.

Let $R$ denote the curvature tensor of $M$ and let $\overline{R}$ denote the curvature tensor of $\overline{M}$.  Denote the un-normalized sectional curvature as $k(A,B):=\lb R(A,B)B,A\rb$ and $\overline{k}(A,B):=\lb \overline{R}(A,B)B,A\rb$.  Denote the restriction of $k$ to $\nu_{q_0}(\Soul)\times \nu_{q_0}(\Soul)$ as $k^f$.  Since a Sharafudtinov fiber is always totally geodesic at a point of the soul, $k^f$ could also be interpreted as the curvature at $q_0$ of the intrinsic metric on the Sharafutinov fiber through $q_0$.  Notice that $k^f$ must be invariant under the action of the isotropy group, $S^1$, on $\nu_{q_0}(\Soul)$.

O'Neill's formula gives
\begin{eqnarray}\label{kff}
k^f(\mathfrak{V},\mathfrak{W}) & = & k(\overline{\mathfrak{V}},\overline{\mathfrak{W}}) + 3|A(\overline{\mathfrak{V}},\overline{\mathfrak{W}})|^2 = 3|A(\overline{\mathfrak{V}},\overline{\mathfrak{W}})|^2 \\
  & = & 3\left|\left(-\hat{R}+S,0,\hat{R}-\hat{L}-S,\hat{R}-S,-\hat{R},S\right)\right|^2\notag
\end{eqnarray}
where $\hat{R}:=\text{Im}(W\overline{V})$, $\hat{L}:=\text{Im}(\overline{V}W)$ and $S:=\frac 14\left(\lb 3\hat{R}-\hat{L},\jj\rb\jj +\lb 3\hat{R}-\hat{L},\kk\rb\kk\right)$.

Equation~\ref{kff} yields the following explicit vertical curvature formula:
\begin{prop}\label{P:verticurv} Identifying $M\cong(S^1\bs S^3)\times\HH$ and $\Soul\cong\{([q],0)\mid q\in S^3\}$ and $q_0\cong([1],0)$ as before, the unnormalized sectional curvature of the vectors $V=a+b\ii+c\jj+d\kk$ and $W=x+y\ii+z\jj+w\kk$ in $\nu_{q_0}(\Soul)\cong\HH$ equals:
\begin{eqnarray*}
k^f(V,W) & = & 6x^2c^2+6z^2a^2+6x^2d^2+6w^2a^2+21d^2z^2+21c^2w^2 \\ & & 9z^2b^2+9w^2b^2+9y^2d^2+9y^2c^2+9b^2x^2+9a^2y^2\\
             &   &-42dzcw-18wbyd-18yczb-18bxay \\ & & -12aydz+12aycw-12bxcw+12bxdz-12xcza-12xdwa.
\end{eqnarray*}
\end{prop}

\section{Verifying the soul inequality}
As reviewed in the introduction, nonnegative curvature implies that the soul inequality is satisfied; that is:
$$\lb (D_XR^\nabla)(V),W\rb^2 \leq\left(\lb R^\nabla(V),W\rb^2 + \frac 23 (D_XD_X k^f)(W,V)\right)\cdot 16,$$
for all $p\in\Soul$, all unit-length $X\in T_p\Soul$ and all $V,W\in \nu_p(\Soul)$.

Denote the right side minus the left side of this inequality as $\text{IN}(X,V,W)$.  It suffices to understand this inequality when $p=q_0=([1],0)$ and $X=2\jj$, since any other $(p,X)$ can be taken to $(q_0,2\jj)$ by the isometric $S^3$ action on $M$.
\begin{prop}\label{IIN} If $q_0=([1],0)\in\Soul$, $X=2\jj\in T_{q_0}\Soul$ and $V,W\in\nu_{q_0}(\Soul)\cong\HH$ have components $V=a+b\ii+c\jj+d\kk$ and $W=x+y\ii+z\jj+w\kk$, then
\begin{eqnarray*}
\text{IN}(2\jj,V,W) & = &
 188b^2z^2+55d^2x^2+512b^2x^2+512a^2y^2+1104d^2z^2 \\
 & &+1104c^2w^2+55a^2w^2+188c^2y^2\\
 & &+1300bxdz-1408bxcw -1408aydz+1300aycw-2208dzcw \\
 & &-1024bxay-376cybz+108cydx+108awbz-110awdx.
\end{eqnarray*}
\end{prop}
\begin{proof}
Lemma~\ref{mncurv} gives:
$$R^\nabla(V) = \frac 72 V\ii - \frac{15}{2}\ii V = 4b -4a\ii + 11d\jj - 11c\kk.$$
Lemma~\ref{L:DR} gives:
$$(D_X R^\nabla)(V) = \frac{15}{8}(X\ii - \ii X)V - \frac{21}{8}V(X\ii-\ii X) = -3d + 18c\ii - 18b\jj + 3a\kk.$$
It remains to compute the expression $(D_XD_X k^f)(W,V)$.  Consider the path
\begin{eqnarray*}
V(t) & = &  e^{-\frac 14tX}\cdot V\cdot e^{-\frac 34tX}\\
       & = & \left(1-\frac 14tX+\frac{1}{32}t^2X^2+\cdots\right)\cdot V\cdot\left(1-\frac 34tX+\frac{9}{32}t^2X^2+\cdots\right) \\
       & = & V - t\left(\frac 14 XV +\frac 34 VX\right)
        + t^2\left( \frac{9}{32}VX^2+\frac{1}{32} X^2V+\frac{3}{16}XVX\right)+\cdots
\end{eqnarray*}
(the terms of order greater than 2 are not exhibited here because they will not effect the final answer).  Define the path $W(t)$ analogously.  Corollary~\ref{parallel} implies that $e^{-tX}\star V(t)$ and $e^{-tX}\star W(t)$ are parallel along $t\mapsto\exp(tX)$.  Define $f(t) = k^f(V(t),W(t))$, which can be explicitly computed using Proposition~\ref{P:verticurv}. Since the $S^3$ action on $M$ is by isometries, we have:
\begin{eqnarray*}
(D_XD_X k^f)(W,V) & = & f''(0) \\ & = & -48bxay+156dzcw-12dxaw-96cybz+24a^2y^2-78d^2z^2\\
& & -78c^2w^2+6d^2x^2+48c^2y^2+48b^2z^2+6a^2w^2+24b^2x^2.
\end{eqnarray*}
The result follows by combining terms and simplifying.
\end{proof}

We now use the above explicit formula to demonstrate that $\text{IN}$ only equals $0$ when it is forced to do so by Corollary~\ref{C:splitting}:
\begin{cor}\label{Marius} Let $\{X,Y\}$ be an orthonormal basis of $\text{span}\{\jj,\kk\}$, and consider the splitting: $\HH=\sigma_1\oplus\sigma_2$, where $\sigma_1=\text{span}\{1,X\}$ and $\sigma_2=\text{span}\{\ii,Y\}$.  If $V,W\in\nu_{q_0}(\Soul)\cong\HH$, then $\text{IN}(2X,V,W) \geq 0$ and $=0$ if and only if $V$ and $W$ are linearly dependent or $\text{span}\{V,W\}$ equals $\sigma_1$ or $\sigma_2$.
\end{cor}
\begin{proof}
Due to the isometry group, it suffices to verify the case $X=\jj$.  For this, let $A=bx-ay$, $B=dz-cw$, $C=bz-cy$ and $D=dx-aw$.  The equation of Proposition~\ref{IIN} simplifies to:
$$\text{IN}(2\jj,V,W) = 107 C^2 + 19D^2 + (9C-6D)^2 + 28 A^2 + 108 B^2 + (22A+32B)^2.$$
This expression is nonnegative and equals zero if and only if $A=B=C=D=0$, which occurs if and only if $V$ and $W$ are linearly dependent or $\text{span}\{V,W\}$ equals $\text{span}\{1,\jj\}$ or equals $\text{span}\{\ii,\kk\}$.
\end{proof}

\begin{cor}
The vector $V\in\nu_{q_0}(\Soul)\cong\HH$ is good if and only if $V$ is perpendicular to neither $1$ nor $\ii$.
\end{cor}
\begin{proof}
If $V$ is perpendicular to neither $1$ nor $\ii$, then Corollary~\ref{Marius} implies that $\text{IN}(X,V,W)>0$ for all non-zero $X$ and all $W$ not parallel to $V$, which means that $V$ is good.  If $V$ is perpendicular to $1$, then we can choose $X$ parallel to the $(\jj,\kk)$-component of $V$, and choose $W$ so that $\text{span}\{V,W\} = \text{span}\{\ii,X\}$.  It then follows from Corollary~\ref{Marius} (or from Corollary~\ref{C:splitting}) that $\text{IN}(X,V,W)=0$, so $V$ is not good.  Similarly, if $V$ is perpendicular to $\ii$, then we can choose $X$ parallel to the $(\jj,\kk)$-component of $V$ and choose $W$ so that $\text{span}\{V,W\}=\text{span}\{1,X\}$, so that $\text{IN}(X,V,W)=0$.
\end{proof}
Wilking described explicitly the locus of point at which zero-curvature planes occur for his metric.  The previous corollary implies that the good vectors are exactly the vectors which exponentiate to positive curvature points of Wilking's metric.  Thus, the soul inequality (which is based on second derivative information at the soul) contains complete information about which points of the sphere bundle have positive curvature.
\section{Metrics with Nonnegative Curvature on $S^3\times\R^3$}
Since Wilking's metric on $S^2\times S^3$ also extends to the vector bundle $S^3\times\R^3$, we mention here that this extension could never be altered to make the soul inequality become strict.  Much more generally:
\begin{prop}
No structures on a vector bundle over an odd dimensional base space could strictly satisfy the soul inequality.
\end{prop}
\begin{proof}Let $B$ denote the base space of the bundle.  Chose $p\in B$ and orthogonal unit-length vectors $W,V$ in the fiber at $p$ such that $k^f(W,V)$ is maximal (among all such $p,V,W$), which implies that $(D_XD_X k^f)(W,V)=0$ for all $X\in T_p B$.  Since $X\mapsto R^\nabla(W,V)X$ is a skew-symmetric endomorphism of the odd-dimensional vector space $T_p B$, there exists a non-zero vector $X\in T_p B$ such that $R(W,V)X=0$.  For any $Y\in T_p B$, the right side of the soul inequality vanishes for the vectors $\{X,Y,W,V\}$.
\end{proof}

\bibliographystyle{amsplain}

\end{document}